\documentclass[12pt,reqno]{amsart}
\usepackage{enumerate}
\usepackage{amsrefs}
\usepackage{amsfonts, amsmath, amssymb, amscd, amsthm, bm, cancel}
\usepackage{url}
\usepackage{graphicx}
\usepackage[
linktocpage=true,colorlinks,citecolor=magenta,linkcolor=blue,urlcolor=magenta]{hyperref}
\usepackage{multicol}
\usepackage{comment}
\usepackage{mathrsfs}
\usepackage[T1]{fontenc}
\usepackage[margin=1in]{geometry}
\usepackage{appendix}

\makeatletter
\@namedef{subjclassname@2020}{\textup{2020} Mathematics Subject Classification}
\makeatother

\usepackage{caption}
\usepackage{tikz}
\usetikzlibrary{arrows.meta}
\usetikzlibrary{chains}
\usetikzlibrary{patterns}

 \newtheorem{thm}{Theorem}[section]
  
 \newtheorem{cor}[thm]{Corollary}
  
 \newtheorem{lem}[thm]{Lemma}
 
 \newtheorem{prop}[thm]{Proposition}
 \theoremstyle{definition}
 
  \newtheorem*{ack}{Acknowledgments}
\newtheorem{prob}{Problem}[section]
 \theoremstyle{remark}
 \newtheorem{rem}[thm]{Remark}

 \numberwithin{equation}{section}
 \allowdisplaybreaks

\newcommand{\g}{\varphi}

\begin{document}
\title[The isotropic horospherical $p$-Minkowski problem in hyperbolic plane]{Classification of solutions to the isotropic\\
horospherical $p$-Minkowski problem in hyperbolic plane}

\author{Haizhong Li}
\address{Department of Mathematical Sciences, Tsinghua University, Beijing 100084, P.R. China}
\email{\href{mailto:lihz@tsinghua.edu.cn}{lihz@tsinghua.edu.cn}}

\author{Yao Wan}
\address{Department of Mathematical Sciences, Tsinghua University, Beijing 100084, P.R. China}
\email{\href{mailto:y-wan19@mails.tsinghua.edu.cn}{y-wan19@mails.tsinghua.edu.cn}}



\begin{abstract}
In \cite{LX}, the first author and Xu introduced and studied the horospherical $p$-Minkowski problem in hyperbolic space $\mathbb{H}^{n+1}$. In particular, they established the uniqueness result for solutions to this problem when the prescribed function is constant and $p\ge -n$. This paper focuses on the isotropic horospherical $p$-Minkowski problem in hyperbolic plane $\mathbb{H}^{2}$, which corresponds to the equation
\begin{equation}\label{0}
     \varphi^{-p}\left(\varphi_{\theta\theta}-\frac{\varphi_{\theta}^2}{2\varphi}+\frac{\varphi-\varphi^{-1}}{2}\right)=\gamma\quad\text{on}\ \mathbb{S}^1,
\end{equation}
where $\gamma$ is a positive constant. We provide a classification of solutions to (\ref{0}) for $p\ge -7$, as well as a nonuniqueness result of solutions to (\ref{0}) for $p<-7$. Furthermore, we extend this problem to the isotropic horospherical $q$-weighted $p$-Minkowski problem in hyperbolic plane and derive some uniqueness and nonuniqueness results.
\end{abstract}

\maketitle

\section{Introduction}

The classical Brunn-Minkowski theory focuses on the convex geometry and geometric analysis of convex bodies by use of the Minkowski sum, including geometric measures and geometric inequalities of convex bodies, see, for example Schneider's book \cite{Sch14}. In recent decades, this theory has been extended to the $L_p$ Brunn-Minkowski theory and the dual Brunn-Minkowski theory, see e.g. \cites{CW06,HLY16,HZ18,HLY05,Lut75,Lut93,Lut96,LYZ04,LYZ18} and references therein. Furthermore, the $L_p$ dual Minkowski problem is a common generalization of the $L_p$ Minkowski problem and the dual Minkowski problem. Building upon the work of Ben Andrews \cite{Ben03}, who classified the solutions to the planar isotropic $L_p$ Minkowski problem, we provided a classification of solutions to the planar isotropic $L_p$ dual Minkowski problem in \cite{LW22}.

For the theory of convex geometry in hyperbolic space, the first author and Xu \cite{LX} introduced the \textit{hyperbolic $p$-sum}, a sum of two sets in hyperbolic space for any $p>0$. This leads to the development of the horospherical $p$-Brunn-Minkowski theory in hyperbolic space. The authors defined the horospherical $k$-th $p$-surface area measure associated with a smooth horospherically convex bounded domain in hyperbolic space, and proposed the following \textit{horospherical $p$-Minkowski problem} for $k=0$:
\begin{prob}[The horospherical $p$-Minkowski problem]\label{prob-horospherical p-Minkowski problem}
    Let $p$ be a real number. Given a smooth positive function $f(z)$ defined on $\mathbb{S}^n$, what are necessary and sufficient conditions for $f$, such that there exists a smooth uniformly h-convex bounded domain $K \subset \mathbb{H}^{n+1}$ for which its horospherical $p$-surface area measure has density $f$, i.e.
	\begin{align*}
	d S_{p}(K,z) = f(z)d\sigma. 
	\end{align*} 
	That is, finding a smooth positive solution to
	\begin{align*}
	\g^{-p}(z) \det (A[\g(z)]) =f(z),
	\end{align*}
	such that the matrix
    \begin{align*}
        A_{ij}[\g(z)]=D_j D_i \g - \frac{1}{2} \frac{|D \g|^2}{\g} \sigma_{ij} + \frac{1}{2} \left(\g- \frac{1}{\g}\right)\sigma_{ij} >0
    \end{align*} 
    for all $z \in \mathbb{S}^n$, where $\sigma_{ij}$ denotes the standard metric on $\mathbb{S}^n$ and $D$ denotes the Levi-Civita connection with respect to $\sigma_{ij}$.
\end{prob}

By using a new volume preserving flow, the first author and Xu \cite{LX} solved the existence of solutions to the horospherical $p$-Minkowski problem for all $p\in(-\infty,+\infty)$ when the given measure is even. Afterwards, the horospherical $p$-Brunn-Minkowski theory was further investigated by \cites{LW23,LW23-10,HWZ}, etc.

In particular, when $f(z)$ is constant and $n=1$, this problem reduces to the \textit{isotropic horospherical $p$-Minkowski problem in hyperbolic plane}, which corresponds to the following second-order ordinary differential equation
\begin{equation}\label{1.1}
    \varphi^{-p}\left(\varphi_{\theta\theta}-\frac{\varphi_{\theta}^2}{2\varphi}+\frac{\varphi-\varphi^{-1}}{2}\right)=\gamma\quad \text{on}\ \mathbb{S}^1,
\end{equation}
where $\gamma$ is a positive constant. In \cite{LX}, the first author and Xu obtained the uniqueness of solutions to (\ref{1.1}) for $p\ge -1$.

\begin{thm}[\cite{LX}]\label{Thm1.1}
\begin{enumerate}
	\item If $p>1$, then the solution to (\ref{1.1}) must be constant, i.e. $\varphi(\theta)\equiv c$, where $c$ solves the equation
    \begin{equation}\label{1.2}
		c^{1-p}-c^{-1-p}=2\gamma.
	\end{equation}
     Denote $\gamma_0 =\frac{(p-1)^{\frac{p-1}{2}}}{(p+1)^{\frac{p+1}{2}}}$. When $0<\gamma<\gamma_0$, (\ref{1.2}) has exactly two solutions. When $\gamma=\gamma_0$, (\ref{1.2}) has a unique solution. When $\gamma>\gamma_0$, (\ref{1.2}) has no solution.
 
	\item If $p =1$, then the solution to (\ref{1.1}) must be constant that satisfies (\ref{1.2}). When $0<\gamma<\frac{1}{2}$, (\ref{1.2}) has a unique solution. When $\gamma\ge \frac{1}{2}$, (\ref{1.2}) has no solution.
 
	\item If $-1<p< 1$, then the solution to (\ref{1.1}) must be constant that satisfies (\ref{1.2}). For each $\gamma>0$, (\ref{1.2}) has a unique solution. 
    
	\item If $p=-1$, then the solutions to (\ref{1.1}) are given by
    \begin{equation}\label{1.3}
        \varphi(\theta)=\alpha\cos\left(\theta-\theta_0\right)+\beta,
    \end{equation}
    where $\theta_0\in\mathbb{R},\ \alpha\in(-\sqrt{1+2\gamma},\sqrt{1+2\gamma})$ and $\beta=\sqrt{1+2\gamma-\alpha^2}$. 
\end{enumerate}
\end{thm}

In this paper, by applying the method employed in \cites{Ben03,LW23}, we establish the following uniqueness of solutions to (\ref{1.1}) for $-7\le p<-1$, and the nonuniqueness of solutions to (\ref{1.1}) for $p<-7$ and $\gamma>\gamma_p$.
\begin{thm}\label{Thm1.2}
\begin{enumerate}[(i)]
    \item If $-7\le p< -1$, then the unique solution to (\ref{1.1}) is $\varphi(\theta)\equiv c$, where $c$ solves the equation (\ref{1.2}).

    \item If $-17\le p< -7$ and $\gamma>\gamma_{p}=\frac{-3}{1+p}\left(\frac{7+p}{1+p}\right)^{\frac{p-1}{2}}$, then there exist at least two solutions to (\ref{1.1}). 
    
    \item Generally, let $k$ and $l$ be two integers such that $1\le l< k$. If $1-2(k+1)^2\le p<1-2k^2$ and $\gamma>\gamma_{p,l}$, where $\gamma_{p,l}$ is given by
    \begin{align*}
        \gamma_{p,l}
        =\frac{1-(l+1)^2}{1+p}\left(\frac{2(l+1)^2-1+p}{1+p}\right)^{\frac{p-1}{2}},
    \end{align*}
then there exist at least $(l+1)$ solutions to (\ref{1.1}). 
Furthermore, for any integer $2\le j\le l+1$, there exists a non-constant solution that corresponds to a h-convex curve $\Gamma_{j,p,\gamma}$ with $j$-fold symmetry in $\mathbb{H}^2$.
\end{enumerate}
\end{thm}

\begin{figure}[htbp]
\centering
\begin{tikzpicture}[scale=0.8][>=Stealth] 

\draw[->](-18,0)--(2.3,0) node[right]{$p$-axis};
\draw[->](0,-1)--(0,6.3) node[above]{$\gamma$-axis};
\foreach \x in {-17, -16, -15, -14, -13, -12, -11, -10, -9, -8, -7, -6, -5, -4, -3, -2, -1, 1, 2} \draw (\x, 1pt) -- (\x, -1pt) node[anchor=north] {$\x$};
\foreach \y in {-1, 1} \draw (1pt, \y) -- (-1pt, \y) node[anchor=east] {$\y$};
\node[below left] at (0,3) {$10$}; \node[below left] at (0,4) {$20$};
\node[below left] at (0,5) {$30$}; \node[below left] at (0,6) {$40$};
\draw [thick, draw=black!100] [domain = -0.15: 0.15] plot ({\x},{2+\x}); 
\draw [thick, draw=black!100] [domain = -0.15: 0.15] plot ({\x},{1.8+\x}); 
\node[below left] at (0,0) {$0$};
\draw [ultra thick, draw=blue] [domain = 1: 2] plot ({\x},{0.313/(\x-0.374)}); 
\filldraw[opacity=0.6, draw=red!100, fill=red!100, domain=1:2] plot(\x,{0.313/(\x-0.374)})--(2,0)--(1,0)--cycle;
\filldraw[opacity=0.6, draw=white, fill=black!20, domain=1:2] plot(\x,{0.313/(\x-0.374)})--(2,6)--(1,6)--cycle;
\filldraw[opacity=0.6, fill=red!100] (1,0.5) circle (0.02);
\filldraw[opacity=0.6, draw=black!20, fill=black!100] (1,0.5)--(1,6);
\filldraw[opacity=0.6, draw=white, fill=blue!100] (-1,0)--(1,0)--(1,6)--(-1,6)--cycle;
\draw [ultra thick, draw=yellow!100] [domain = 0: 6] plot ({-1},{\x});
\filldraw[opacity=0.6, draw=white, fill=blue!100] (-7,0)--(-1,0)--(-1,6)--(-7,6)--cycle;
\node[above] at (-7,6) {{\footnotesize{$p=-7$}}};
\node[above] at (-11.74,6) {{\footnotesize{$\gamma_p=\frac{-3}{1+p}\left(\frac{7+p}{1+p}\right)^{\frac{p-1}{2}}$}}};
\filldraw[opacity=0.6, draw=blue, fill=blue!100, domain=-17.5: -11.74] plot({\x},{-8.6/(\x+10.3)})--(-7,6)--(-7,0)--(-17.5,0)--cycle;
\filldraw[ultra thick, pattern color=black, pattern=north east lines, opacity=0.4, draw=blue!40, domain=-17.5: -11.74] plot({\x},{-8.6/(\x+10.3)})--(-7,6)--(-7,0)--(-17.5,0)--cycle;
\filldraw[opacity=0.6, draw=blue!0, fill=red!100, domain=-17.5: -11.74] plot({\x},{-8.6/(\x+10.3)})--(-11.74,6)--(-17.5,6)--cycle;
\filldraw[ultra thick, pattern color=black, pattern=north east lines, opacity=0.4, draw=blue!40, domain=-17.5: -11.74] plot({\x},{-8.6/(\x+10.3)})--(-11.74,6)--(-17.5,6)--cycle;
\draw[opacity=0.6, draw=blue!100] (-7,0)--(-7,6);
\end{tikzpicture}
\caption{Number $k$ of solutions to (\ref{1.1})}
\caption*{Grey domain: $k=0$. Blue domain: $k=1$. Red domain: $k=2$. Blue domain with shadow: $k\ge1$. Red domain with shadow: $k\ge2$. Yellow line: $k=\infty$.}
\label{fig-1}
\end{figure}


$\ $

In a similar manner as in the case of Euclidean spaces, by modifying the powers of weight $\cosh r=\frac{|D \g|^2}{2\g}+\frac{\g+\g^{-1}}{2}$, we naturally extend Problem \ref{prob-horospherical p-Minkowski problem} to the \textit{horospherical $q$-weighted $p$-Minkowski problem}:
\begin{prob}[The horospherical $q$-weighted $p$-Minkowski problem]\label{prob q-weighted-horospherical p-Minkowski problem}
	Let $p$ and $q$ be two real numbers. 
	Given a positive function $f(z)$ defined on $\mathbb{S}^n$, what are necessary and sufficient conditions for $f$, such that there exists a smooth positive solution $\g(z)$ satisfying
	\begin{align*}
	\g^{-p}(z)\left(\frac{|D \g|^2}{2\g}(z) +\frac{\g(z)+\g^{-1}(z)}{2} \right)^{q-n} \det (A[\g(z)]) =f(z),
	\end{align*}
	such that $A[\g(z)]>0$ for all $z \in \mathbb{S}^n$.
\end{prob}

Note that when $q=n$, Problem \ref{prob q-weighted-horospherical p-Minkowski problem} reduces to the horospherical $p$-Minkowski problem; when $q=n+1$, Problem \ref{prob q-weighted-horospherical p-Minkowski problem} reduces to Problem 11.1 in \cite{LX}.

$\ $

Now we consider the \textit{isotropic horospherical $q$-weighted $p$-Minkowski problem in hyperbolic plane}, which corresponds to the equation
\begin{equation}\label{1.4}
    \varphi^{-p}\left(\frac{\varphi_{\theta}^2}{2\varphi}+\frac{\varphi+\varphi^{-1}}{2}\right)^{q-1}\left(\varphi_{\theta\theta}-\frac{\varphi_{\theta}^2}{2\varphi}+\frac{\varphi-\varphi^{-1}}{2}\right)=\gamma
    \quad \text{on}\ \mathbb{S}^1,
\end{equation}
where $\gamma$ is a positive constant. In this paper, we establish the following uniqueness result.
\begin{thm}\label{Thm1.3}
Suppose either
\begin{enumerate}[(i)]
    \item $p\ge -1,\ q\le 1$ and at least one of these inequalities is strict, or  
    \item $p< -1,\ q\ge 1$ and $q-p\le 8$.  
\end{enumerate}
Then the solution to (\ref{1.4}) must be constant, i.e. $\varphi(\theta)\equiv c$, where $c$ solves the equation
\begin{equation*}
    c^{-p}\left(\frac{c+c^{-1}}{2}\right)^{q-1}\left(\frac{c-c^{-1}}{2}\right)=\gamma.
\end{equation*}
\end{thm}

On the other hand, we also prove the nonuniqueness of solutions to (\ref{1.4}).
\begin{thm}\label{Thm1.4}
Let $l\ge 1$ be an integer. If $p<-1,\ q\ge 1,\ q-p > 2(l+1)^2$ and $\gamma>\gamma_{p,q,l}$, where $\gamma_{p,q,l}$ is given by
\begin{align*}
    \gamma_{p,q,l}= 2^{-q} (u_{\gamma}^{2}+u_{\gamma}^{-2})^{q-1} (u_{\gamma}^{2-2p}-u_{\gamma}^{-2-2p}),
\end{align*}
where $u_{\gamma}\in(1,\infty)$ solves the equation
\begin{align}\label{1.5}
    (1-p)+(1+p)u_{\gamma}^{-4}+(q-1)\frac{(1-u_{\gamma}^{-4})^2}{1+u_{\gamma}^{-4}}=2(l+1)^2,
\end{align}
then there exist at least $(l+1)$ solutions to (\ref{1.4}). Furthermore, for any integer $2\le j\le l+1$, there exists a non-constant solution that corresponds to a h-convex curve $\Gamma_{j,p,q,\gamma}$ with $j$-fold symmetry in $\mathbb{H}^2$.
\end{thm}

\begin{figure}[htbp]
\centering
\begin{tikzpicture}[scale=0.5][>=Stealth] 

\draw[->](-8.3,0)--(8.3,0) node[right]{{\footnotesize{$p$-axis}} };
\draw[->](0,-2.3)--(0,8.3) node[above]{{\footnotesize{$q$-axis}} };
\foreach \x in {-8, -7, -6, -5, -4, -3, -2, -1, 1, 2, 3, 4, 5, 6, 7, 8 } \draw (\x, 1pt) -- (\x, -1pt) node[anchor=north] {{\footnotesize{$\x$}} };
\foreach \y in { -2, -1, 1, 2, 3, 4, 5, 6, 7, 8} \draw (1pt, \y) -- (-1pt, \y) node[anchor=east] {{\footnotesize{$\y$}} };
\node[below right] at (0,0) {{\footnotesize{$0$}} };
\draw [densely dashed] (-8,1)--(8,1);
\draw [densely dashed] (-1,-2)--(-1,8);
\draw [densely dashed] (-8,-0)--(0.5,8.5);
\filldraw[opacity=0.6, draw=black!20, fill=black!20] (-1,1)--(8,1)--(8,-2.2)--(-1,-2.2)--cycle;
\node[right] at (8,1) {{\footnotesize{$q=1$}}};
\filldraw[opacity=0.6, draw=black!20, fill=black!20] (-1,1)--(-1,7)--(-7,1)--cycle;
\node[below] at (-1,-2) {{\footnotesize{$p=-1$}}};
\filldraw[opacity=0.8, draw=black!20, fill=black!90] (-1,7)--(-1,8)--(-8,8)--(-8,1)--(-7,1)--cycle;
\node[below right] at (0.5,8.5) {{\footnotesize{$q-p=8$}}};
\end{tikzpicture}
\caption{Whether or not there exist non-constant solutions to (\ref{1.4})}
\caption*{Grey domain: Non-existence for $\gamma>0$. Black domain: Existence for $\gamma>\gamma_{p,q,1}$. }
\label{fig-2}
\end{figure}

\begin{rem}
Let $\psi(u_{\gamma})$ denote the left hand side of (\ref{1.5}). Since $p<-1$ and $q\ge 1$, then $\psi$ is strictly increasing with respect to $u_\gamma$. Note that $\psi$ tends to $2<2(l+1)^2$ as $u_{\gamma}$ tends to 1, and also tends to $q-p>2(l+1)^2$ as $u_{\gamma}$ tends to infinity.  Therefore, there exists a unique solution $u_{\gamma}$ to (\ref{1.5}), and $\gamma_{p,q,l}$ is well-defined.
\end{rem}

The paper is organized as follows. In Section \ref{sec:2}, we construct a period function $\Theta_{p}(\gamma,E)$ which is closely related to the solution to (\ref{1.1}), and give two alternative expressions $\Theta_{p}[u_{-},u_{+}]$ and $\Theta_{p}\{\alpha,r\}$. In Section \ref{sec:3}, we study special values and asymptotic behaviours of $\Theta_{p}$. In Section \ref{sec:4}, we prove that when $p\le -1$, then $\Theta_p\{\alpha,r\}$ is monotone increasing in $p$ and in $\alpha$, respectively. In Section \ref{sec:5}, we prove Theorem \ref{Thm1.2}. In Section \ref{sec:6}, by using the same method to the equation (\ref{1.4}), we give the proofs of Theorem \ref{Thm1.3} and Theorem \ref{Thm1.4}.

\begin{ack}
This work was supported by NSFC Grant No.11831005.
\end{ack}


\section{The period function}
\label{sec:2}

Let $\varphi(\theta)$ be a smooth positive solution to the isotropic horospherical $p$-Minkowski problem in hyperbolic plane, that is, $\varphi(\theta)$ satisfies
\begin{equation}\label{2.1}
    \varphi^{-p}\left(\varphi_{\theta\theta}-\frac{\varphi_{\theta}^2}{2\varphi}+\frac{\varphi-\varphi^{-1}}{2}\right)=\gamma,
\end{equation}
where $\gamma$ is a positive constant. Due to Theorem \ref{Thm1.1}, we only consider the case where $p\le -1$ and $u(\tau)$ is not constant from now on.

Define $u(\tau)=\sqrt{\varphi(2\tau)}$, then (\ref{2.1}) is equivalent to
\begin{equation}\label{2.2}
    u^{1-2p}(u_{\tau\tau}+u-u^{-3})=2\gamma.
\end{equation}
Since $p<0$, then (\ref{2.2}) can be rewritten as 
\begin{equation*}
\begin{aligned}
     &  (u_{\tau}^2+u^2+u^{-2})_{\tau}=\frac{2\gamma}{p}(u^{2p})_{\tau}.
\end{aligned}
\end{equation*}
Thus a first integral of (\ref{2.2}) is given by
\begin{equation}\label{2.3}
\begin{aligned}
    &  u_{\tau}^2+u^2+u^{-2}-\frac{2\gamma}{p}u^{2p}=E,
\end{aligned}
\end{equation}
for some constant $E$.

Since the solution $\varphi$ to (\ref{2.1}) is a $2\pi$-periodic positive function of $\theta$, then the corresponding solution $u$ to (\ref{2.2}) is a $\pi$-periodic positive function of $\tau$, and its maximum $u_{+}$ and minimum values $u_{-}$ are determined by $p,\gamma$ and $E$. Define a function
$$E_{\gamma}(u)=u^2+u^{-2}-\frac{2\gamma}{p}u^{2p},$$
then it follows from (\ref{2.3}) that $E=E_{\gamma}(u_{\pm})$. Note that
$$E_{\gamma}'(u)=2u-2u^{-3}-4\gamma u^{2p-1},$$
and
$$E_{\gamma}''(u)=2+6u^{-4}-4\gamma(2p-1) u^{2p-2}>0,$$
thus $E_{\gamma}(u)$ is convex, and decreasing for $u\in (0, u_{\gamma})$ and increasing for $u\in (u_{\gamma},+\infty)$, where $u_{\gamma}$ satisfies $E_{\gamma}'(u_{\gamma})=0$. Moreover, $E_{\gamma}(u)$ has the minimum value $E_{\gamma}^{*}=E_{\gamma}(u_\gamma)$, and $E_{\gamma}(u)$ tends to infinity as $u$ tends to zero or infinity. Consequently, we obtain

\begin{lem}\label{Lem 2.1}
Let $u_{+}>u_{-}$ be the two positive roots of the equation $E_{\gamma}(u)=E$, where $E>E_{\gamma}^{*}$. Then $u_{+}$ increases to infinity and $u_{-}$ decreases to zero as $E$ increases to infinity.

Furthermore, if $\gamma_1>\gamma_2>0$ and $E>E_{\gamma_1}^*,E> E_{\gamma_2}^*$, denote $u_{+}^{(i)},u_{-}^{(i)}$ by the two roots of the equation $E_{\gamma_i}(u)=E$ for $i=1,2$, then $u_{+}^{(2)}>u_{+}^{(1)}>u_{-}^{(1)}>u_{-}^{(2)}>0$.
\end{lem}

Let $\tau_{-},\tau_{+}\in [0,\pi)$ be two consecutive critical points such that $u(\tau_{\pm})=u_{\pm}$. Without loss of generality, we assume $\tau_{+}>\tau_{-}$. Using (\ref{2.3}), we have 
\begin{equation*}
    u_{\tau}=\sqrt{E+\frac{2\gamma}{p}u^{2p}-u^2-u^{-2}}>0
\end{equation*}
for $\tau\in (\tau_{-},\tau_{+})$. Then
\begin{equation}\label{2.4}
    \tau_{+}-\tau_{-}=\int_{\tau_{-}}^{\tau_{+}}d\tau
    =\int_{u_{-}}^{u_{+}}\frac{du}{u_{\tau}}.
\end{equation}
Thus we consider the period function
\begin{equation}\label{2.5}
    \Theta_p(\gamma,E):=\int_{u_{-}}^{u_{+}}\frac{du}{\sqrt{E+\frac{2\gamma}{p}u^{2p} -u^2-u^{-2}}},
\end{equation}
where $u_{-}$ and $u_{+}$ satisfy $E_{\gamma}(u_{-})=E_{\gamma}(u_{+})=E$.

If $u(\tau)$ is a $\pi$-periodic non-constant solution to (\ref{2.2}), then it must have an even number of critical points in $[0,\pi)$. It follows from (\ref{2.4}) and (\ref{2.5}) that there exists a positive integer $m$ such that $\Theta_p(\gamma,E)=\frac{\pi}{2m}$. We obtain

\begin{lem}\label{Lem 2.2}
\begin{enumerate}[(i)]
    \item If $\frac{\pi}{4}<\Theta_p(\gamma,E)<\frac{\pi}{2}$ for any $E>E_{\gamma}^*$, then the solution to (\ref{2.1}) must be constant. Moreover, the solution to (\ref{2.1}) is unique.
    
    \item If $E>E_{\gamma}^*$ and $\Theta_p(\gamma,E)=\frac{\pi}{2m}$ for some integer $m$, then there exists a $\frac{2\pi}{m}$-periodic non-constant solution to (\ref{2.1}).
\end{enumerate}
\end{lem}

For convenience, we provide two new expressions of $\Theta_p$. First, using $E_{\gamma}(u_{-})=E_{\gamma}(u_{+})=E$, we get
\begin{equation}\label{2.6}
    -\frac{2\gamma}{p}=\frac{(u_{+}^2-u_{-}^2)(u_{+}^{-2} u_{-}^{-2}-1)}{u_{+}^{2p}- u_{-}^{2p}},
\end{equation}
and
\begin{equation}\label{2.7}
    E=\frac{u_{+}^{2p}(u_{-}^2+u_{-}^{-2})-u_{-}^{2p}(u_{+}^2+u_{+}^{-2})}{u_{+}^{2p}- u_{-}^{2p}}.
\end{equation}
It follows from $p\le -1$ and (\ref{2.6}) that $\gamma>0$ if and only if $u_{+}u_{-}>1$. We can then consider $\Theta_p$ as a function of $u_{-}$ and $u_{+}$
\begin{equation}\label{2.8}
    \Theta_p [u_{-},u_{+}]:=\Theta_p(\gamma(u_{-},u_{+}),E(u_{-},u_{+}))
    =\int_{u_{-}}^{u_{+}}\frac{du}{\sqrt{G_p(u,u_{-},u_{+})}},
\end{equation}
where
\begin{equation}\label{2.9}
\begin{split}
    G_p(u,u_{-},u_{+})
    =\frac{u_{-}^{2p}-u^{2p}}{u_{-}^{2p}-u_{+}^{2p}}(u_{+}^2+u_{+}^{-2})
    +\frac{u^{2p}-u_{+}^{2p}}{u_{-}^{2p}-u_{+}^{2p}}(u_{-}^2+u_{-}^{-2})
    -u^2-u^{-2}
\end{split}
\end{equation}
for $(u_{-},u_{+})$ satisfying $u_{+}>u_{-}>0$ and $u_{+}u_{-}>1$.

Next, let $r=\frac{u_{+}}{u_{-}}\in(1,\infty)$ and $\alpha=\frac{1}{u_{+}u_{-}}\in (0,1)$, and make the change of variable $s=\frac{u}{u_{-}}$. We define
\begin{equation}\label{2.10}
    \Theta_{p}\{\alpha,r\}:=\Theta_p[\alpha^{-\frac{1}{2}}r^{-\frac{1}{2}},\alpha^{-\frac{1}{2}}r^{\frac{1}{2}}]
    =\int_{1}^r \frac{ds}{\sqrt{F_p(s,\alpha,r)}},
\end{equation}
where
\begin{equation}\label{2.11}
\begin{split}
    F_p(s,\alpha,r)
    &=\frac{1-s^{2p}}{1-r^{2p}}(r^2+\alpha^{2})
    +\frac{s^{2p}-r^{2p}}{1-r^{2p}}(1+\alpha^2 r^2)
    -s^2-\alpha^2 r^2 s^{-2}\\
    &=\frac{1-s^{2p}}{1-r^{2p}}(r^2-1)-(s^2-1)
    -\alpha^2\left[\frac{1-s^{2p}}{1-r^{2p}}(r^{2}-1)-r^2(1- s^{-2})\right]
\end{split}
\end{equation}
for $(\alpha,r)$  satisfying $0<\alpha<1<r$.

In the following sections, we will study the asymptotic behavior and monotonicity of the period function $\Theta_p$.


\section{Special values and asymptotic behaviors}
\label{sec:3}

Let $p\le -1$. Denote $D_1=\{(\gamma,E):\ \gamma>0,\ E>E_{\gamma}^{*}\}$ and $D_2=\{(\alpha,r):\ 0<\alpha<1<r\}$.

\begin{thm}\label{Thm3.1}
$\Theta_p(\cdot,\cdot)$ and $\Theta_p\{\cdot,\cdot\}$ are continuous on $D_1$ and $D_2$, respectively. $\Theta_p$ is continuous with respect to $p$, and has following limiting values:
\begin{align}
    & \label{3.1}
    \lim\limits_{p\to -\infty}\Theta_p\{\alpha,r\}=\arccos\sqrt{\frac{1-\alpha^2}{r^2-\alpha^2}},\\
    & \label{3.2}
    \lim\limits_{r\to 1}\Theta_p\{\alpha,r\}=\frac{\pi}{\sqrt{(2-2p)+(2+2p)\alpha^2}},\\
    & \label{3.3}
    \lim\limits_{r\to +\infty}\Theta_p\{\alpha,r\}=\frac{\pi}{2},\\
    & \label{3.4}
    \lim\limits_{\gamma\to 0}\Theta_p(\gamma,E)=\frac{\pi}{2},\\
    & \label{3.5}
    \lim\limits_{E\to +\infty}\Theta_p(\gamma,E)=\frac{\pi}{2},\\
    & \label{3.6}
    \lim\limits_{E\to (E_{\gamma}^*)^{+}}\Theta_p(\gamma,E)=\frac{\pi}{\sqrt{(2-2p)+(2+2p)u_{\gamma}^{-4}}},
\end{align}
where $u_{\gamma}\in(1,+\infty)$ is given by
\begin{equation}\label{3.7}
    u_{\gamma}^{2-2p}-u_{\gamma}^{-2-2p}=2\gamma.
\end{equation}

Moreover, we also have
\begin{equation}\label{3.8}
    \Theta_{-1}(\gamma,E)=\frac{\pi}{2}, \quad\quad (\gamma,E)\in D_1,
\end{equation}
\begin{equation}\label{3.9}
    \Theta_{p}\{1,r\}=\frac{\pi}{2}, \quad \quad\quad\quad r>1.
\end{equation}
\end{thm}

The following lemma is needed in the proof of (\ref{3.6}), see \cite[Corollary 4.2]{Per10}.
\begin{lem}[\cite{Per10}]\label{Lem3.2}
Let $\delta$ be a positive real number and let $f:(u_0-\delta,u_0+\delta)\to \mathbb{R}$ be a smooth function such that $f(u_0)=f'(u_0)=0$ and $f''(u_0)=-2a<0$. If for any small $c>0$, $u_1(c)<u_0<u_2(c)$ are such that $f(u_1(c))+c=0=f(u_2(c))+c$, then
\begin{equation*}
    \lim\limits_{c\to 0^+}\int_{u_1(c)}^{u_2(c)}\frac{du}{\sqrt{f(u)+c}}=\frac{\pi}{\sqrt{a}}.
\end{equation*}
\end{lem}

For the sake of completeness, let us give a proof of Lemma \ref{Lem3.2}.
\begin{proof}[Proof of Lemma \ref{Lem3.2}]
For any $b>a$, we define the function $h(u)=f'(u)+2b(u-u_0)$. Since $h'(u_0)=2(b-a)>0$, there exists a $\epsilon(b)\in (0,\delta)$ such that $h'(u) > 0$ for any $u\in[u_0,u_0+\epsilon]$. Then it follows from $h(u_0)=0$ that $h(u)>0$ for $u\in(u_0,u_0+\epsilon]$.

For any small $c>0$ such that $u_2(c) <u_0+\epsilon$, the function
$$g(u)=f(u)+c+b(u-u_2(c))(u+u_2(c)-2u_0)$$
satisfies $g(u_2(c)) = 0$ and $g'(u) = h(u) > 0$. Thus $g(u) < 0$ for $u\in[u_0,u_2(c))$. We get
$$0 < f(u) + c < b(u_2(c)-u)(u+u_2(c)-2u_0),\quad \text{for any}\ u\in(u_0,u_2(c))$$
and then
$$\frac{\pi}{2\sqrt{b}}=\int_{u_0}^{u_2(c)}\frac{du}{\sqrt{b(u_2(c)-u)(u+u_2(c)-2u_0)}}<\int_{u_0}^{u_2(c)}\frac{du}{\sqrt{f(u)+c}}.$$

On the other hand, for any $0<b<a$ and small $c>0$, the same argument shows
$$\int_{u_0}^{u_2(c)}\frac{du}{\sqrt{f(u)+c}}<\int_{u_0}^{u_2(c)}\frac{du}{\sqrt{b(u_2(c)-u)(u+u_2(c)-2u_0)}}=\frac{\pi}{2\sqrt{b}}.$$
Therefore, we obtain $ \lim\limits_{c\to 0^+}\int_{u_0}^{u_2(c)}\frac{du}{\sqrt{f(u)+c}}=\frac{\pi}{2\sqrt{a}}.$
The similar calculation holds for $\int_{u_1(c)}^{u_0}\frac{du}{\sqrt{f(u)+c}}$, and then we complete the proof of this lemma.
\end{proof}

\begin{proof}[Proof of Theorem \ref{Thm3.1}]
At first, it follows from (\ref{2.10}) and (\ref{2.11}) that
\begin{align*}
    \lim\limits_{p\to -\infty}\Theta_p\{\alpha,r\}
    =\int_1^r\frac{ds}{\sqrt{(r^2-s^2)(1-\alpha^2 s^{-2})}}
    =\arccos\sqrt{\frac{1-\alpha^2}{r^2-\alpha^2}}.
\end{align*}

Now let $\beta>0$, and set $s^\beta=x$, where
\begin{align*}
    v=\frac{r^\beta-1}{2}z+\frac{r^\beta+1}{2},\quad\quad z\in (-1,1).
\end{align*}
Then the change of variables formula implies that
\begin{align*}
    \Theta_p\{\alpha,r\}=\int_{-1}^1\frac{dz}{\sqrt{J_p(\alpha,r,\beta,z)}},
\end{align*}
where
\begin{equation*}
    J_p(\alpha,r,\beta,z)
    =\frac{4\beta^2}{(r^\beta-1)^2}x^{\frac{2(\beta-1)}{\beta}}
    \left[\frac{1-x^{\frac{2p}{\beta}}}{1-r^{2p}}(r^2+\alpha^2)
    +\frac{x^{\frac{2p}{\beta}}-r^{2p}}{1-r^{2p}}(1+\alpha^2 r^2)
    -x^{\frac{2}{\beta}}-\alpha^2 r^2 x^{-\frac{2}{\beta}}\right].
\end{equation*}

When $r$ is close to $1$, a direct computation using Taylor expansions yields
\begin{align*}
\begin{split}
    J_p(\alpha,r,\beta,z)
    &=\left((2-2p)+(2+2p)\alpha^2\right)(1-z^2) \\
    &\quad \cdot \left[1+\frac{z}{2}\left(\beta-\frac{4-2p}{3}-\frac{4(1+p)\alpha^2}{(3-3p)+(3+3p)\alpha^2}\right)(r-1)+o(r-1)\right],
\end{split}
\end{align*}
from which it follows that
\begin{align*}
    \lim\limits_{r\to 1}\Theta_p\{\alpha,r\}
    =\int_{-1}^1\frac{dz}{\sqrt{\left((2-2p)+(2+2p)\alpha^2\right)(1-z^2)}}=\frac{\pi}{\sqrt{(2-2p)+(2+2p)\alpha^2}}.
\end{align*}

When $r$ tends to infinity, let $y=\frac{z+1}{2}\in (0,1)$ and choose $\beta=1$, we have
\begin{equation*}
    J_p(\alpha,r,1,z)=4\left[1-y^2+(2-2y)\frac{1}{r}+o\left(\frac{1}{r}\right)\right],\quad p<-1,
\end{equation*}
from which it follows that
\begin{equation*}
    \lim\limits_{r\to +\infty}\Theta_p\{\alpha,r\}=\int_0^1 \frac{dy}{\sqrt{1-y^2}}=\frac{\pi}{2}.
\end{equation*}

Next, let $w=u_{+}-u_{-}$ and $u=w t+u_{-}$ for $t\in (0,1)$. By $E_{\gamma}(u_{\pm})=E$, we have
\begin{equation*}
    \frac{u_{+}}{\sqrt{E}}\to 1,\quad  
    \frac{u_{-}}{\left(\frac{-p}{2\gamma}E\right)^{\frac{1}{2p}}}\to 1\ \text{and}\ 
    \frac{w}{\sqrt{E}}\to 1, \quad \text{as}\ E\to +\infty,
\end{equation*}
from which it follows that
\begin{align*}
    \lim\limits_{E\to +\infty}\Theta_p(\gamma,E)
    &=\lim\limits_{E\to +\infty}\int_0^1 \frac{dt}{\sqrt{\frac{E}{w^2}-\left(t+\frac{u_{-}}{w}\right)^2-w^{-4}\left(t+\frac{u_{-}}{w}\right)^{-2}+\frac{2\gamma}{p}w^{2p-2}\left(t+\frac{u_{-}}{w}\right)^{2p}}}\\
    &=\int_0^1 \frac{d t}{\sqrt{1-t^2}}
    =\frac{\pi}{2}.
\end{align*}

On the other hand, we choose $f(u)=E_{\gamma}^*+\frac{2\gamma}{p}u^{2p}-u^2-u^{-2}$ and $c=E-E_{\gamma}^*$ in Lemma \ref{Lem3.2}, then $u_0=u_{\gamma},\ u_{1}(c)=u_{-}$ and $ u_{2}(c)=u_{+}$. It follows from $f'(u_{\gamma})=0$ that $u_{\gamma}$ satisfies (\ref{3.7}). Thus we have 
\begin{align*}
    a &=-\frac{1}{2}f''(u_{\gamma})=1+3u_{\gamma}^{-4}+2\gamma (2p-1)u_{\gamma}^{2p-2}\\
    &=(2-2p)+(2+2p)u_{\gamma}^{-4}.
\end{align*}
Using Lemma \ref{Lem3.2}, we obtain
\begin{align*}
    \lim\limits_{E\to (E_{\gamma}^*)^{+}}\Theta_p(\gamma,E)
    =\frac{\pi}{\sqrt{a}}=\frac{\pi}{\sqrt{(2-2p)+(2+2p)u_{\gamma}^{-4}}}.
\end{align*}

Finally, (\ref{3.4}), (\ref{3.8}) and (\ref{3.9}) can be deduced by direct calculation.
\end{proof}


\section{Monotonicity of $\Theta_p\{\alpha,r\}$}
\label{sec:4}

Let $p\le -1$ and $0<\alpha<1<r$. By (\ref{2.10}) and (\ref{2.11}), we have
\begin{equation}\label{4.1}
    \Theta_{p}\{\alpha,r\}
    =\int_{1}^r \frac{ds}{\sqrt{F_p(s,\alpha,r)}},
\end{equation}
where
\begin{align}
    F_p(s,\alpha,r)
    &=\frac{s^{2p}-1}{r^{2p}-1}(r^2+\alpha^{2})
    +\frac{r^{2p}-s^{2p}}{r^{2p}-1}(1+\alpha^2 r^2)
    -s^2-\alpha^2 r^2 s^{-2}   \label{4.2}\\
    &=\frac{s^{2p}-1}{r^{2p}-1}(r^2-1)-(s^2-1)
    -\alpha^2\left[\frac{s^{2p}-1}{r^{2p}-1}(r^{2}-1)-r^2(1- s^{-2})\right]. \label{4.3}
\end{align}

\subsection{Monotonicity in $p$}$\ $

Given $1<s<r$, we denote
\begin{equation*}
    H(p)=\frac{s^{2p}-1}{r^{2p}-1}\in (0,1).
\end{equation*}

\begin{lem}\label{Lem4.1}
$H(p)$ is monotone decreasing in $p<0$.
\end{lem}
\begin{proof}
Taking the derivative of $H(p)$ yields
\begin{align*}
    H'(p)=\frac{1}{p (r^{2p}-1)^3(s^{2p}-1)}\left(\frac{s^{2p}\log s^{2p}}{s^{2p}-1}-\frac{r^{2p}\log r^{2p}}{r^{2p}-1}\right).
\end{align*}
Let $h(t)=\frac{t\log t}{t-1}$ for $0<t<1$. Taking the derivative of $h(t)$ yields
\begin{align*}
    h'(t)=\frac{t-1-\log t}{(t-1)^2}> 0.
\end{align*}
Since $0<r^{2p}<s^{2p}<1$, we have $h(r^{2p})<h(s^{2p})$ and therefore $H'(p)<0$. We complete the proof of Lemma \ref{Lem4.1}.
\end{proof}

Since $(r^2-1)(1-\alpha^2)>0$ for $0<\alpha<1<r$, by (\ref{4.1}) and (\ref{4.3}), Lemma \ref{Lem4.1} immediately deduces the following theorem.
\begin{thm}\label{Thm4.2}
Let $p\le -1$. $\Theta_p\{\alpha,r\}$ is monotone increasing in $p$.
\end{thm}

This together with (\ref{3.8}) yields an upper bound for $\Theta_p\{\alpha,r\}$.

\begin{cor}\label{Cor4.3}
Let $p<-1$. $\Theta_p\{\alpha,r\}<\frac{\pi}{2}$ for any $0<\alpha<1<r$.
\end{cor}

\subsection{Monotonicity in $\alpha$}$\ $

Using Lemma \ref{Lem4.1}, we have
\begin{equation}\label{4.4}
    \frac{s^{2p}-1}{r^{2p}-1}(r^2-1)-(s^2-1)>0,
\end{equation}
and
\begin{equation}\label{4.5}
    \frac{s^{2p}-1}{r^{2p}-1}(r^{2}-1)-r^2(1- s^{-2})>0,
\end{equation}
for $1<s<r$ and $p<-1$. It follows from (\ref{4.3}), (\ref{4.4}) and (\ref{4.5}) that $F_p(s,\alpha,r)$ is monotone decreasing in $\alpha\in(0,1)$ for any $1<s<r$ and $p<-1$. Therefore, we obtain

\begin{thm}\label{Thm4.4}
Let $0<\alpha<1$. $\Theta_p\{\alpha,r\}$ is monotone increasing in $\alpha$.
\end{thm}

Combining this with (\ref{3.8}), we can conclude Corollary \ref{Cor4.3} again. Moreover, denote
\begin{equation*}
    \hat{\Theta}_p(r):=\Theta_p\{0,r\}=\int_1^r \frac{ds}{\sqrt{\frac{s^{2p}-1}{r^{2p}-1}(r^2-1)-(s^2-1)}}.
\end{equation*}
Making the change of variable $s^{\beta}=x=\frac{r^{\beta}-1}{2}z+\frac{r^{\beta}+1}{2}$ with $\beta=\frac{4-2p}{3}$ yields
\begin{equation*}
   \hat{\Theta}_p(r)=\int_{-1}^1 \frac{dz}{\sqrt{\hat{J}_p(r,z)}},
\end{equation*}
where 
\begin{equation*}
    \hat{J}_p(r,z)=\frac{4\beta^2}{(r^{\beta}-1)^2}\left(\frac{r^2-r^{2p}}{1-r^{2p}}x^{\frac{2(\beta-1)}{\beta}}-x^2-\frac{r^2-1}{1-r^{2p}}x^{\frac{2(\beta-1+p)}{\beta}}\right).
\end{equation*}
Note that $\hat{\Theta}_p(r)$ is the period function $\Theta\left(\frac{1}{1-2p},r\right)$ given by Ben Andrews \cite{Ben03}. We recall the properties of $\hat{\Theta}_p(r)$ as follows:
\begin{prop}[\cite{Ben03}]\label{Prop4.5}
Let $p\le -1$. Then $\hat{\Theta}_p(\cdot)$ is continuous on $(1,\infty)$, and has the following limiting values:
\begin{align}
    & \label{4.6}
    \lim\limits_{r\to\infty}\hat{\Theta}_p(r)=\frac{\pi}{2},\\
    & \label{4.7}
    \lim\limits_{r\to 1}\hat{\Theta}_p(r)=\frac{\pi}{\sqrt{2-2p}}.
\end{align}
\end{prop}

\begin{prop}[\cite{Ben03}]\label{Prop4.6}
Let $p<-1$ and $r>1$. Then
$\frac{\partial}{\partial r}\hat{J}_p(r,z)<0$ for any $z\in(-1,1)$. Furthermore, $\hat{\Theta}_p(r)$ is monotone increasing in $r$.   
\end{prop}

Consequently, we deduce the following corollary.
\begin{cor}\label{Cor4.7}
Let $p<-1$. $\Theta_p\{\alpha,r\}>\hat{\Theta}_p(r)>\frac{\pi}{\sqrt{2-2p}}$ for any $0<\alpha<1<r$.
\end{cor}


\section{Proof of Theorem \ref{Thm1.2}}
\label{sec:5}

At first, if $-7\le p<-1$, using Corollary \ref{Cor4.3} and Corollary \ref{Cor4.7}, we have
\begin{align*}
    \frac{\pi}{2}>\Theta_{p}\{\alpha,r\}>\frac{\pi}{\sqrt{2-2p}}>\frac{\pi}{4},
\end{align*}
for any $0<\alpha<1<r$. Therefore, if $\gamma>0$ and $E>E_{\gamma}^*$, then
\begin{align*}
    \Theta_p(\gamma,E)=\Theta_{p}\{\alpha(\gamma,E),r(\gamma,E)\}\in \left(\frac{\pi}{4},\frac{\pi}{2}\right).
\end{align*}
It follows from Lemma \ref{Lem 2.2} (i) that the solution to (\ref{1.1}) must be constant. Moreover, the solution to (\ref{1.1}) is unique. We prove Theorem \ref{Thm1.2} (i).

Next, if $1-2(k+1)^2\le p<1-2k^2$ and $\gamma>\gamma_{p,l}$ for $1\le l<k$, then it follows from (\ref{3.5}), (\ref{3.6}) and (\ref{3.7}) that
\begin{align*}
    \lim\limits_{E\to +\infty}\Theta_p(\gamma,E)=\frac{\pi}{2}
\end{align*}
and
\begin{align*}
    \lim\limits_{E\to (E_{\gamma}^{*})^{+}}\Theta_p(\gamma,E)<\frac{\pi}{2(l+1)}.
\end{align*}
Therefore, there exist at least $l$ values $\{E_j\}_{j=1}^{l}$ such that $E_j>E_{\gamma}^*$ and $\Theta_p(\gamma,E_j)=\frac{\pi}{2(j+1)}$. It follows from Lemma \ref{Lem 2.2} (ii) that there exist at least $(l+1)$ solutions to (\ref{1.1}). We complete the proof of Theorem \ref{Thm1.2}.


\section{Proofs of Theorem \ref{Thm1.3} and Theorem \ref{Thm1.4}}
\label{sec:6}

\subsection{Proof of Theorem \ref{Thm1.3} (i)}$\ $

Let $\varphi(\theta)$ be the solution to (\ref{1.4}), i.e.
\begin{equation}\label{6.1}
    \varphi^{-p}\left(\frac{\varphi_{\theta}^2}{2\varphi}+\frac{\varphi+\varphi^{-1}}{2}\right)^{q-1}\left(\varphi_{\theta\theta}-\frac{\varphi_{\theta}^2}{2\varphi}+\frac{\varphi-\varphi^{-1}}{2}\right)=\gamma,\quad \text{on}\ \mathbb{S}^1,
\end{equation}
where $\gamma$ is a positive constant. The case $q=1$ of Theorem \ref{Thm1.3} (i) was shown in Theorem \ref{Thm1.1}. Recall that the following Heintze-Karcher type inequality in $\mathbb{H}^2$ is crucial in the proof of Theorem \ref{Thm1.1} in \cite{LX}.

\begin{lem}[\cite{LX}]\label{Lem-HK}
    Let $\Omega$ be a smooth uniformly h-convex bounded domain in $\mathbb{H}^2$, and $\varphi(\theta)$ be its exponential horospherical support function. Then we have
    \begin{equation}\label{6.2}
        \int_{\partial\Omega}\left(\frac{\cosh r-\Tilde{u}}{\Tilde{\kappa}}-\Tilde{u}\right)d\mu\ge 0,
    \end{equation}
    which is equivalent to
    \begin{equation}\label{6.3}
        \int_{\mathbb{S}^1}\left(\varphi_{\theta\theta}-\frac{\varphi_{\theta}^2}{\varphi}\right) \left(\varphi_{\theta\theta}-\frac{\varphi_{\theta}^2}{2\varphi}+\frac{\varphi-\varphi^{-1}}{2}\right) d\theta\ge 0.
    \end{equation}
    Equality holds if and only if $\Omega$ is a geodesic ball, i.e. $\varphi(\theta)$ is given by (\ref{1.3}).
\end{lem}

Now we can give the proof. If $p\ge -1$ and $q< 1$, by (\ref{6.1}) and (\ref{6.3}), we have
\begin{align*}
    0 &\le \int_{\mathbb{S}^1}\left(\varphi_{\theta\theta}-\frac{\varphi_{\theta}^2}{\varphi}\right) \varphi^{p} \left(\frac{\varphi_{\theta}^2}{2\varphi}+\frac{\varphi+\varphi^{-1}}{2}\right)^{1-q} d\theta\\
    &=\int_{\mathbb{S}^1}\left(\frac{\varphi_{\theta}}{\varphi}\right)_{\theta} \varphi^{1+p} \left(\frac{\varphi_{\theta}^2}{2\varphi}+\frac{\varphi+\varphi^{-1}}{2}\right)^{1-q} d\theta\\
    &=- (1+p)\int_{\mathbb{S}^1}\varphi_{\theta}^2 \varphi^{p-1} \left(\frac{\varphi_{\theta}^2}{2\varphi}+\frac{\varphi+\varphi^{-1}}{2}\right)^{1-q} d\theta\\
    &\quad  -(1-q)\int_{\mathbb{S}^1}\varphi_{\theta}^2 \varphi^{1+p} \left(\varphi,\frac{\varphi_{\theta}^2}{2\varphi}+\frac{\varphi+\varphi^{-1}}{2}\right)^{-q}\left(\varphi_{\theta\theta}-\frac{\varphi_{\theta}^2}{2\varphi}+\frac{\varphi-\varphi^{-1}}{2}\right) d\theta\\
    & \le 0.
\end{align*}
Hence the above inequalities are both equalities, and then $\varphi_{\theta}\equiv0 $ on $\mathbb{S}^1$. This completes the proof of  Theorem \ref{Thm1.3} (i).

\subsection{The period function $\Theta_{p,q}$}$\ $

By Theorem \ref{Thm1.3} (i), we only consider the case where $p\le -1,\ q\ge 1$ and $\varphi(\theta)$ is not constant from now on. Define $u(\tau)=\sqrt{\varphi(2\tau)}$, then (\ref{6.1}) is equivalent to
\begin{equation}\label{6.4}
    u^{1-2p}(u_{\tau}^2+u^2+u^{-2})^{q-1} (u_{\tau\tau}+u-u^{-3})=2^q \gamma.
\end{equation}
Since $p<0<q$, (\ref{6.4}) can be rewritten as 
\begin{equation*}
\begin{aligned}
    &  \left(\frac{1}{q}(u_{\tau}^2+u^2+u^{-2})^{q}\right)_{\tau}=\frac{2^q \gamma}{p}(u^{2p})_{\tau}.
\end{aligned}
\end{equation*}
Thus a first integral of (\ref{6.4}) is given by
\begin{equation*}
\begin{aligned}
    &  \frac{1}{q}(u_{\tau}^2+u^2+u^{-2})^{q}-\frac{2^q\gamma}{p}u^{2p}=E.
\end{aligned}
\end{equation*}
for some constant $E$.

For the case $p\le -1$ and $q\ge 1$, the argument is roughly the same as the case $q=1$. Let $\bar{E}_{\gamma}(u)=\frac{1}{q} (u^2+u^{-2})^{q}-\frac{2^q\gamma}{p}u^{2p}$, it follows that $\bar{E}_{\gamma}(u)$ is convex, and decreasing for $u\in (0, u_{\gamma})$ and increasing for $u\in (u_{\gamma},+\infty)$, where $u_{\gamma}$ satisfies $\bar{E}_{\gamma}'(u_{\gamma})=0$. Moreover, $\bar{E}_{\gamma}(u)$ has the minimum value $\bar{E}_{\gamma}^{*}=\bar{E}_{\gamma}(u_\gamma)$, and $\bar{E}_{\gamma}(u)$ tends to infinity as $u$ tends to zero or infinity.

For any $E>\bar{E}_{\gamma}^{*}$, we define the period function
\begin{equation}\label{6.5}
    \Theta_{p,q}(\gamma,E):=\int_{u_{-}}^{u_{+}}\frac{du}{u_{\tau}}
    =\int_{u_{-}}^{u_{+}}\frac{du}{\sqrt{\left(q E+\frac{2^q q\gamma}{p}u^{2p}\right)^{\frac{1}{q}}-u^2-u^{-2}}},
\end{equation}
where $u_{-}$ and $u_{+}$ satisfy $\bar{E}_{\gamma}(u_{-})=\bar{E}_{\gamma}(u_{+})=E$. Since $u(\tau)$ is $\pi$-periodic, we obtain

\begin{lem}\label{Lem 6.2}
\begin{enumerate}[(i)]
    \item If $\frac{\pi}{4}<\Theta_{p,q}(\gamma,E)<\frac{\pi}{2}$ for any $E>\bar{E}_{\gamma}^*$, then the solution to (\ref{6.1}) must be constant. Moreover, the solution to (\ref{6.1}) is unique.
    
    \item If $E>\bar{E}_{\gamma}^*$ and $\Theta_{p,q}(\gamma,E)=\frac{\pi}{2m}$ for some integer $m$, then there exists a $\frac{2\pi}{m}$-periodic non-constant solution to (\ref{6.1}).
\end{enumerate}
\end{lem}

We reparameterize $\Theta_{p,q}$. Let $r=\frac{u_{+}}{u_{-}}>1$ and $\alpha=\frac{1}{u_{+}u_{-}}$. It follows from $\bar{E}_{\gamma}(u_{\pm})=E$ that
\begin{equation}\label{6.6}
    -\frac{2^q q\gamma}{p}=\frac{(u_{+}^2+u_{+}^{-2})^q-(u_{-}^2+u_{-}^{-2})^q}{u_{-}^{2p}-u_{+}^{2p}}
    =\frac{(r^2+\alpha^2)^q-(1+\alpha^2 r^2)^q}{(\alpha r)^{q-p} (1-r^{2p})},
\end{equation}
and
\begin{equation}\label{6.7}
    q E=\frac{u_{-}^{2p}(u_{+}^2+u_{+}^{-2})^q-u_{+}^{2p}(u_{-}^2+u_{-}^{-2})^q}{u_{-}^{2p}- u_{+}^{2p}}
    =\frac{(r^2+\alpha^2)^q-r^{2p} (1+\alpha^2 r^2)^q}{(\alpha r)^{q} (1-r^{2p})}.
\end{equation}
Since $p\le -1,\ q\ge 1$ and (\ref{6.6}), we can conclude that $\gamma>0$ if and only if $0<\alpha<1$. Making the change of variable $s=\frac{u}{u_{-}}$, we define
\begin{equation}\label{6.8}
    \Theta_{p,q}\{\alpha,r\}:=\Theta_{p,q}(\gamma(\alpha,r),E(\alpha,r))
    =\int_{1}^r \frac{ds}{\sqrt{F_{p,q}(s,\alpha,r)}},
\end{equation}
where
\begin{equation}\label{6.9}
\begin{split}
    F_{p,q}(s,\alpha,r)
    &=\left[\frac{1-s^{2p}}{1-r^{2p}}(r^2+\alpha^{2})^q +\frac{s^{2p}-r^{2p}}{1-r^{2p}}(1+\alpha^{2}r^2)^q\right]^{\frac{1}{q}}-s^2-\alpha^2 r^2 s^{-2}
\end{split}
\end{equation}
for $(\alpha,r)$ satisfying $0<\alpha<1<r$.


\subsection{Properties of $\Theta_{p,q}$} $\ $

In this subsection, we assume $p\le -1$ and $q\ge 1$. Denote $\bar{D}_1=\{(\gamma,E):\ \gamma>0,\ E>\bar{E}_{\gamma}^{*}\}$ and $\bar{D}_2=\{(\alpha,r):\ 0<\alpha<1,\ r>1\}$. Since the proof of the following properties of $\Theta_{p,q}$ follows similarly, we briefly present it. 

\begin{thm}\label{Thm6.3}
$\Theta_{p,q}(\cdot,\cdot)$ and $\Theta_{p,q}\{\cdot,\cdot\}$ are continuous on $\bar{D}_1$ and $\bar{D}_2$, respectively. $\Theta_{p,q}$ is continuous with respect to $p$ and $q$, and has following limiting values:
\begin{align}
    & \label{6.10}
    \lim\limits_{p\to -\infty}\Theta_{p,q}\{\alpha,r\}=\arccos\sqrt{\frac{1-\alpha^2}{r^2-\alpha^2}},\\
    & \label{6.11}
    \lim\limits_{q\to +\infty}\Theta_{p,q}\{\alpha,r\}=\arccos\sqrt{\frac{1-\alpha^2}{r^2-\alpha^2}},\\
    & \label{6.12}
    \lim\limits_{r\to 1}\Theta_{p,q}\{\alpha,r\}=\frac{\pi}{\sqrt{(2-2p)+(2+2p)\alpha^2+(2q-2)\frac{(1-\alpha^2)^2}{1+\alpha^2} }},\\
    & \label{6.13}
    \lim\limits_{r\to +\infty}\Theta_{p,q}\{\alpha,r\}=\frac{\pi}{2},\\
    & \label{6.14}
    \lim\limits_{\gamma\to 0}\Theta_{p,q}(\gamma,E)=\frac{\pi}{2},\\
    & \label{6.15}
    \lim\limits_{E\to +\infty}\Theta_{p,q}(\gamma,E)=\frac{\pi}{2},\\
    & \label{6.16}
    \lim\limits_{E\to (\bar{E}_{\gamma}^*)^+}\Theta_{p,q}(\gamma,E)=\frac{\pi}{\sqrt{ (2-2p)+(2+2p) u_{\gamma}^{-4}+(2q-2)\frac{(1-u_{\gamma}^{-4})^2}{1+u_{\gamma}^{-4}}}},
\end{align}
where $u_{\gamma}\in(1,\infty)$ is given by
\begin{equation}\label{6.17}
    u_{\gamma}^{2-2p}-u_{\gamma}^{-2-2p}=2^q \gamma ( u_{\gamma}^2+u_{\gamma}^{-2})^{1-q}.
\end{equation}
\end{thm}

We will use the following lemma from \cites{PP15,Per10} in the proof of (\ref{6.16})
\begin{lem}[\cites{PP15,Per10}]\label{Lem6.4}
Let $f(u,E)$ and $g(u,E)$ be smooth functions such that
$$f(u_0,E_0)=\frac{\partial f}{\partial u}(u_0,E_0)=0 \quad \text{and}\quad \frac{\partial^2 f}{\partial u^2}(u_0,E_0)=-2A<0.$$
Suppose that $u_{-}=u_{-}(E)$ and $u_{+}=u_{+}(E)$ satisfy
$$f(u_{-},E)=f(u_{+},E)=0,\quad u_{-}<u_0<u_{+},$$
with
$$f(u,E)>0\ \text{for all}\ u\in (u_{-},u_{+})\quad \text{and}\quad 
u_{-},u_{+}\to u_0\ \text{as}\ E\to (E_0)^+.$$
Then we have
\begin{align*}
    \lim\limits_{E\to (E_0)^+}\int_{u_{-}}^{u_{+}}\frac{g(u,E)du}{\sqrt{f(u,E)}}=g(u_0,E_0)\frac{\pi}{\sqrt{A}}.
\end{align*}
\end{lem}

\begin{proof}[Proof of Theorem \ref{Thm6.3}]

We will only provide the proofs of (\ref{6.12}) and (\ref{6.16}), as the others can be derived similarly to the proof of Theorem \ref{Thm3.1}.

Let $\beta>0$, and set $s^\beta=x$, where
\begin{align*}
    x=\frac{r^\beta-1}{2}z+\frac{r^\beta+1}{2},\quad\quad z\in (-1,1).
\end{align*}
Then the change of variables formula implies that
\begin{align*}
    \Theta_{p,q}\{\alpha,r\}=\int_{-1}^1\frac{dz}{\sqrt{J_{p,q}(\alpha,r,\beta,z)}},
\end{align*}
where
\begin{equation*}
\begin{split}
    J_{p,q}(\alpha,r,\beta,z)
    &=\frac{4\beta^2}{(r^\beta-1)^2} x^{\frac{2(\beta-1)}{\beta}} \left[\left(\frac{1-x^{\frac{2p}{\beta}}}{1-r^{2p}}(r^2+\alpha^2)^q+\frac{x^{\frac{2p}{\beta}}-r^{2p}}{1-r^{2p}}(1+\alpha^2 r^2)^q\right)^{\frac{1}{q}}\right.\\
    &\quad\quad \left. -x^{\frac{2}{\beta}}-\alpha^2 r^2 x^{-\frac{2}{\beta}}\right].
\end{split}
\end{equation*}
When $r$ is close to $1$, a direct computation using Taylor expansions yields
\begin{equation*}
\begin{split}
    J_{p,q}(\alpha,r,\beta,z)
    &=\left((2-2p)+(2+2p)\alpha^2+(2q-2)\frac{(1-\alpha^2)^2}{1+\alpha^2} \right) (1-z^2) \\
    &\quad \cdot \left[1+\frac{z}{2}\left(\beta-\frac{2(2q-p)}{3}-\frac{4\alpha^2 \mu_{p,q}(\alpha,z)}{3}\right)(r-1)+o(r-1)\right],
\end{split}
\end{equation*}
where
\begin{align*}
    \mu_{p,q}(\alpha,z)=\frac{(p+1)-(q-1)\left(\frac{4(2+\alpha^2)}{(1+\alpha^2)^2}-(1+p)\frac{3-\alpha^2}{1+\alpha^2}+(q-1)\frac{2(1-\alpha^2)^2}{1+\alpha^2}\right)}{(1-p)+(1+p)\alpha^2+(q-1)\frac{(1-\alpha^2)^2}{1+\alpha^2}}.
\end{align*}
Thus it follows that
\begin{align*}
    \lim\limits_{r\to 1}\Theta_{p,q}\{\alpha,r\}
    &=\int_{-1}^1\frac{dz}{\sqrt{\left((2-2p)+(2+2p)\alpha^2+(2q-2)\frac{(1-\alpha^2)^2}{1+\alpha^2}\right)(1-z^2)}}\\
    &=\frac{\pi}{\sqrt{(2-2p)+(2+2p)\alpha^2+(2q-2)\frac{(1-\alpha^2)^2}{1+\alpha^2}}}.
\end{align*}



On the other hand, we choose $g(u,E)\equiv 1$ and $f(u,E)=\left(q E+\frac{2^q q\gamma}{p}u^{2p}\right)^{\frac{1}{q}}-u^2-u^{-2}$ in Lemma \ref{Lem6.4}, then $u_0=u_{\gamma}$ and $E_0=\bar{E}_{\gamma}^*$. It follows from $f(u_{\gamma},\bar{E}_{\gamma}^*)=\frac{\partial f}{\partial u}(u_{\gamma},\bar{E}_{\gamma}^*)=0$ that $u_{\gamma}$ satisfies (\ref{6.17}). Moreover, we calculate 
\begin{align*}
     & A =-\frac{1}{2}\frac{\partial^2 f}{\partial u^2}(u_{\gamma},\bar{E}_{\gamma}^*)\\
    = & 1+3u_{\gamma}^{-4}-\frac{2^q \gamma}{q}(2p-1)u_{\gamma}^{2p-2}\left(q E+\frac{2^q q\gamma}{p}u^{2p}\right)^{\frac{1-q}{q}}+\frac{2^{2q+1}(q-1)\gamma^2}{q^2}u_{\gamma}^{4p-2}\left(q E+\frac{2^q q\gamma}{p}u^{2p}\right)^{\frac{1-2q}{q}}\\
    = & 1+3u_{\gamma}^{-4}+(1-2p)(1-u_{\gamma}^{-4})+(2q-2)\frac{(u_{\gamma}^{2}-u_{\gamma}^{-2})^2}{u_{\gamma}^{2} (u_{\gamma}^{2}+u_{\gamma}^{-2})},
\end{align*}
where we used the definition of $u_{\gamma}$. Using Lemma \ref{Lem6.4}, we obtain
\begin{align*}
    \lim\limits_{E\to (\bar{E}_{\gamma}^*)^+}\Theta_{p,q}(\gamma,E)
    =\frac{\pi}{\sqrt{A}}
    =\frac{\pi}{\sqrt{ (2-2p)+(2+2p) u_{\gamma}^{-4}+(2q-2)\frac{(1-u_{\gamma}^{-4})^2}{1+u_{\gamma}^{-4}}}}.
\end{align*}

\end{proof}

\begin{rem}
  Since $r$ tends to 1 if and only if $E$ tends to $\bar{E}_{\gamma}^{*}$, and at the same time $\alpha$ tends to $u_{\gamma}^{-2}$, then (\ref{6.12}) and (\ref{6.16}) are essentially equivalent.
\end{rem}


Since (\ref{6.8}), (\ref{6.9}) and $(r^2-1)(1-\alpha^2)>0$, Lemma \ref{Lem4.1} implies the monotonicity in $p$.
\begin{thm}\label{Thm6.5}
Let $p\le -1$ and $q\ge 1$. $\Theta_{p,q}\{\alpha,r\}$ is monotone increasing in $p$.
\end{thm}

$\ $


As for the monotonicity in $q$, we need a lemma. Given $A,B>1$ and $0<\lambda<1$, denote
\begin{equation*}
    K(q)=(\lambda A^q+(1-\lambda) B^q)^{\frac{1}{q}}.
\end{equation*}

\begin{lem}\label{Lem6.6}
$K(q)$ is monotone increasing in $q>0$.
\end{lem}
\begin{proof}
Taking the derivative of $K(q)$ yields
\begin{align*}
    \frac{K'(q)}{K(q)}
    &=\frac{1}{q^2(\lambda A^q+(1-\lambda) B^q)}\left[\lambda A^q\log A^q +(1-\lambda) B^q\log B^q \right.\\
    &\quad\quad \left. -(\lambda A^q+(1-\lambda) B^q)\log(\lambda A^q+(1-\lambda) B^q)\right].
\end{align*}
It follows from the convexity of $k(t)=t\log t$ that $K'(q)>0$. The proof is complete.
\end{proof}

Choosing $A=r^2+\alpha^{2},\ B=1+r^2\alpha^{2}$ and $\lambda=\frac{s^{2p}-1}{r^{2p}-1}$, and using Lemma \ref{Lem6.6}, we obtain 
\begin{thm}\label{Thm6.7}
Let $p\le -1$ and $q\ge 1$. $\Theta_{p,q}\{\alpha,r\}$ is monotone decreasing in $q$.
\end{thm}

$\ $




Next, we prove the monotonicity in $\alpha$. 
\begin{lem}\label{Lem6.8}
Let $p<-1$, $q\ge 1$ and $0<\alpha<1$. $F_{p,q}(s,\alpha,r)$ is monotone decreasing in $\alpha$.
\end{lem}
\begin{proof}
Let $A(\alpha)=r^2+\alpha^2,\ B(\alpha)=1+r^2\alpha^2$ and $\lambda=\frac{s^{2p}-1}{r^{2p}-1}$. Since $0<\alpha<1<s<r$, by (\ref{4.5}), we have $A>B>1$ and
\begin{equation}\label{6.18}
    \lambda r^{-2}+(1-\lambda)<s^{-2}.
\end{equation}

By (\ref{6.9}), fixing $1<s<r$, we define
\begin{equation*}
    L(\alpha)=F_{p,q}(s,\alpha,r)=(\lambda A^q+(1-\lambda) B^q)^{\frac{1}{q}}-s^2-\alpha^2 r^2 s^{-2}.
\end{equation*}
Taking the derivative of $L(\alpha)$ yields
\begin{align*}
    L'(\alpha)=2\alpha r^2\left[(\lambda A^q+(1-\lambda) B^q)^{\frac{1-q}{q}}(\lambda r^{-2} A^{q-1}+(1-\lambda)B^{q-1})-s^{-2}\right].
\end{align*}
Since $q\ge 1$, $1<s<r$ and $A>B>1$, using Lemma \ref{Lem6.6} and (\ref{6.18}), we have
\begin{align*}
    L'(\alpha)
    &\le 2\alpha r^2\left[(\lambda A^{q-1}+(1-\lambda) B^{q-1})^{-1}(\lambda r^{-2} A^{q-1}+(1-\lambda)B^{q-1})- s^{-2}\right]\\
    &= 2\alpha r^2\ \frac{\lambda A^{q-1}(r^{-2}-s^{-2})+(1-\lambda) B^{q-1}(1-s^{-2})}{\lambda A^{q-1}+(1-\lambda) B^{q-1}}\\
    &\le 2\alpha r^2 B^{q-1} \ \frac{\lambda (r^{-2}-s^{-2})+(1-\lambda) (1-s^{-2})}{\lambda A^{q-1}+(1-\lambda) B^{q-1}} \\
    &<0.
\end{align*}
This completes the proof of the lemma.
\end{proof}

Therefore, we obtain
\begin{thm}\label{Thm6.9}
Let $p<-1$, $q\ge 1$ and $0<\alpha<1$. $\Theta_{p,q}\{\alpha,r\}$ is monotone increasing in $\alpha$.
\end{thm}

$\ $

Finally, we derive the upper and lower bound for $\Theta_{p,q}$. Denote
\begin{equation*}
    \hat{\Theta}_{p,q}(r):=\Theta_{p,q}\{0,r\}
    =\int_1^r \frac{ds}{\sqrt{\left(\frac{1-s^{2p}}{1-r^{2p}}r^{2q} +\frac{s^{2p}-r^{2p}}{1-r^{2p}}\right)^{\frac{1}{q}}-s^2}} .
\end{equation*}
Making the change of variable $s^{\beta}=x=\frac{r^{\beta}-1}{2}z+\frac{r^{\beta}+1}{2}$ with $\beta=\frac{4q-2p}{3}$ yields
\begin{equation*}
    \hat{\Theta}_{p,q}(r)=\int_{-1}^1 \frac{dz}{\sqrt{\hat{J}_{p,q}(r,z)}},
\end{equation*}
where 
\begin{equation*}
    \hat{J}_{p,q}(r,z)=\frac{4\beta^2}{(r^{\beta}-1)^2}\left[x^{\frac{2(\beta-1)}{\beta}}\left(\frac{r^{2q}-r^{2p}}{1-r^{2p}}+\frac{1-r^{2q}}{1-r^{2p}}x^{\frac{2p}{\beta}}\right)^{\frac{1}{q}}-x^2\right].
\end{equation*}
Note that $\hat{\Theta}_{p,q}(r)$ is the period function ${\Theta}(2p,2q,r)$ given by our previous paper \cite{LW22}. We have the following properties.
\begin{prop}[\cite{LW22}]\label{Prop6.10}
Let $p\le -1$ and $q\ge 1$. Then $\hat{\Theta}_{p,q}(r)$ is continuous on $(1,\infty)$, and has the following limiting values:
\begin{align}
    & \label{6.19}
    \lim\limits_{r\to\infty}\hat{\Theta}_{p,q}(r)=\frac{\pi}{2},\\
    & \label{6.20}
    \lim\limits_{r\to 1}\hat{\Theta}_{p,q}(r)=\frac{\pi}{\sqrt{2q-2p}}.
\end{align}
\end{prop}

\begin{prop}[\cite{LW22}]\label{Prop6.11}
Let $p< -1$, $q\ge 1$ and $r>1$. Then $\frac{\partial}{\partial r}\hat{J}_{p,q}(r,z)<0$ for any $z\in (-1,1)$. Moreover, $\hat{\Theta}_{p,q}(r)$ is monotone increasing in $r$.    
\end{prop}

This together with Theorem \ref{Thm6.5}, Theorem \ref{Thm6.7} and Theorem \ref{Thm6.9} gives
\begin{cor}\label{Cor6.12}
Let $p<-1$ and $q\ge 1$. Then for any $0<\alpha<1<r$
\begin{equation}\label{6.21}
     \frac{\pi}{2}>\Theta_{p,q}\{\alpha,r\}> \hat{\Theta}_{p,q}(r)>\frac{\pi}{\sqrt{2q-2p}}.
\end{equation}
\end{cor}

\subsection{Proofs of Theorem \ref{Thm1.3} (ii) and Theorem \ref{Thm1.4}}$\ $

If $q\ge 1$ and $q-8\le p<-1$, using Corollary \ref{Cor6.12}, we have
\begin{align*}
    \frac{\pi}{2}>\Theta_{p,q}\{\alpha,r\}>\frac{\pi}{4}
\end{align*}
for any $0<\alpha<1<r$. Therefore, if $\gamma>0$ and $E>\bar{E}_{\gamma}^*$, then
\begin{align*}
    \Theta_{p,q}(\gamma,E)=\Theta_{p,q}\{\alpha(\gamma,E),r(\gamma,E)\}\in \left(\frac{\pi}{4},\frac{\pi}{2}\right).
\end{align*}
It follows from Lemma \ref{Lem 6.2} (i) that the solution to (\ref{6.1}) must be constant. Moreover, the solution to (\ref{6.1}) is unique. We complete the proof of Theorem \ref{Thm1.3} (ii).

If $p<-1,\ q\ge 1$ and $\gamma>\gamma_{p,q,l}$, then it follows from (\ref{6.15}) and (\ref{6.16}) that 
\begin{equation*}
    \lim\limits_{E\to +\infty}\Theta_{p,q}(\gamma,E)=\frac{\pi}{2}
\end{equation*}
and
\begin{equation*}
    \lim\limits_{E\to (\bar{E}_{\gamma}^*)^+}\Theta_{p,q}(\gamma,E)<\frac{\pi}{2(l+1)}.
\end{equation*}
Therefore, there exist at least $l$ values $\{E_j\}_{j=1}^{l}$ such that $E_j>\bar{E}_{\gamma}^*$ and $\Theta_{p,q}(\gamma,E_j)=\frac{\pi}{2(j+1)}$ for $1\le j\le l$. It follows from Lemma \ref{Lem 6.2} (ii) that there exist at least $(l+1)$ solutions to (\ref{6.1}). We complete the proof of Theorem \ref{Thm1.4}.


$\ $

\end{document}